\documentclass[12pt,twoside, final]{amsart}
\UseRawInputEncoding
\pagespan{1}{\pageref*{LastPage}}
\usepackage{etoolbox,lastpage}
\commby{}
\usepackage{amsmath,amsfonts,amssymb,enumerate}
\usepackage{graphicx}		
\usepackage{color}
\usepackage[colorlinks]{hyperref}
 
\newtheorem{theorem}{Theorem}[section]
\newtheorem{proposition}[theorem]{Proposition}
\newtheorem{lemma}[theorem]{Lemma}
\newtheorem{corollary}[theorem]{Corollary}
\theoremstyle{definition}
\newtheorem{definition}[theorem]{Definition}
\newtheorem{example}[theorem]{Example}
\theoremstyle{remark}
\newtheorem{remark}[theorem]{Remark}
\numberwithin{equation}{section}

\usepackage[all]{xy}
 
\begin{document}

 
\title[The $S$-flat topology ]{The $S$-flat topology } 
 
\author[M. Aqalmoun]{Mohamed Aqalmoun$^*$}
\address[Mohamed Aqalmoun]{Department of Mathematics, Higher Normal School, Sidi Mohamed Ben Abdellah University, Fez, Morocco.}
\email{ maqalmoun@yahoo.fr}


  \thanks{$^*$Corresponding author}
%
 
 \maketitle
%

\begin{abstract}
For a commutative ring $R$ with unit $1\ne 0$ and a multiplicatively closed subset $S$ of $R$, we introduce a new topology on the $S$-prime spectrum $\mathrm{Spec}_SR$ of $R$ called the $S$-flat topology. Our aims is to give an algebraic descriptions of the topological properties like compactness, irreducibility, connectivity and noetherianess with respect to this new topology .
\\
\textbf{Keywords:}   $S$-prime ideal, $S$-Zariski topology, $S$-flat topology.  \\
\textbf{MSC(2010):}  13A15 ; 13A99; 54A99.
\end{abstract}
 
\section{\bf Introduction}
In this paper, we focus on only commutative rings with $1\ne 0$. Let $R$ always denote  such a ring. Consider a subset $S$ of $R$. We call $S$ a multiplicatively closed subset of $R$ if $(i)$ $0\not\in S$, $(ii)$ $1\in S$, and $(iii)$ $ss'\in S$ for all $s,s'\in S$. Let $S$ be a multiplicatively closed subset of $R$ and $P$ be an ideal of $R$. Then the ideal $P$ is called an $S$-prime ideal of $R$ if $P\cap S=\emptyset$ and there exists $s\in S$ and whenever $ab\in P$ then $sa\in P$ or $sb\in P$. If we take $S\subseteq u(R)$, where $u(R)$ denotes the set of units in $R$, the notions of $S$-prime ideal  and prime ideal are the same thing. Here, we denote the sets of all prime and all $S$-prime ideals by $\mathrm{Spec}R$ and $\mathrm{Spec}_SR$, respectively. In \cite{Eda}, the authors construct a topology on $\mathrm{Spec}_SR$ called the $S$-Zariski topology, precisely, the collection of closed subsets for this topology are $V_S(I)=\{P\in \mathrm{Spec}_SR\ / sI\subseteq P \text{ for some } s\in S\}$ where $I$ runs trough the set of all ideals of $R$. Note that, if $S\subseteq u(R)$, then the $S$-Zariski and the classical Zariski topologies are equal.\par 
In $1982$, D. E. Dobbs, M. Fontana and I. J. Papick, in \cite{DoFo}, introduce and studied  a new  topology on the prime spectrum $\mathrm{Spec}R$, which behave completely as a dual of the Zariski topology, and where it is discussed its relation with the inverse-order topology studied by Hochster \cite{Hosh}. It is called the flat topology. For this topology the collection of subsets $V(I)$ where $I$ runs through the set of finitely generated ideals of $R$ forms a basis of opens. Many authors have studied the flat topology, and have obtained several results which are useful studying commutative rings, for instance see, \cite{DoFo}, \cite{Tar}, \cite{Fino},\cite{FiCar}.\par 
In this paper, following \cite{DoFo} and \cite{Tar}, we construct a new topology on the $S$-prime spectrum $\mathrm{Spec}_SR$ called the $S$-flat topology. The $S$-flat topology is the coarsest topology on the $S$-prime spectrum $\mathrm{Spec}_SR$ such that the subsets $D_S(f) = \{P\in \mathrm{Spec}_SR\ /\  sf\not\in P \text{ for all } s\in S \}$ are closed where $f$ runs through $R$. Note that, if $S\subseteq u(R)$ then $S$-flat and flat topology are the same. It follows that, by using the $S$-flat topology of a ring, we generalize the flat topology of $R$. In section $2$, we establish a correspondence between $S$-Zariski closed subsets and $S$-radical ideals. In section $3$, we introduce the $S$-flat topology and some basic properties. In section $3$, it is shown that every $S$-Zariski closed subset is an $S$-flat quasi-compact subset, also the irreducible and connected components are characterized.    
\section{\bf The $S$-Zariski topology}
This section is dedicated to some new results about the $S$-Zariski topology. We  start with some  terminologies. Let $R$ be a commutative ring with unit $\ne 0$, and $S$ be a multiplicative subset of $R$. According to \cite{Eda}, if $E$ is a subset of $R$, the $S$-variety of $E$ is defined by $$V_S(E)=\{P\in \mathrm{Spec}_SR/ \ sE\subseteq P\ \text{ for some } s\in S\}$$ 
The following theorem can be found in \cite{Eda}.
\begin{theorem}
The following statements hold.
\begin{enumerate}
\item $V_S(E)=V_S((E))$ for any subset $E$ of $R$, where $(E)$ is the ideal generated by $E$.
\item $V_S(R)=\emptyset$ and $V_S((0))=\mathrm{Spec}_SR$.
\item (Arbitrary intersection) $\cap_{i\in \Delta}V_S(J_i)=V_S(\cup_{i\in \Delta}J_i)$ for any ideals $J_i$ of $R$.
\item (Finite union) $V_S(I)\cup V_S(J)=V_S(I\cap J)=V_S(IJ)$.
\end{enumerate}
\end{theorem}
From the above result, the set $\{V_S(E)\ / E\subseteq R\}$ satisfies all axioms of closed set for a topology of $\mathrm{Spec}_SR$ called the $S$-Zariski topology. \\
For ideal $I$ of $R$, we denote $D_S(I)$ the open subset $$D_S(I)=\mathrm{Spec}_SR-V_S(I)=\{P\in \mathrm{Spec}_SR\ / sI\not\subseteq P\  \text{ for all } s\in S\}$$
The collection of all $D_S(f)$, $f\in R$ form a basis of the $S$-Zariski on $\mathrm{Spec}_SR$, see \cite{Eda}.
\begin{definition}
Let $I$ be an ideal of $R$. The $S$-radical of $I$ is defined by
$$\sqrt[S]{I}=\{a\in R/ \ sa^n\in I \text{ for some } s\in S \text{ and } n\in \Bbb N\}$$
\end{definition}
\begin{proposition}\label{Variety}
Let $I$ be an ideal of $R$. Then $\sqrt[S]{I}=\displaystyle\cap_{I\subseteq P, S\cap P=\emptyset} P $. In particular if $Q$ is a prime ideal disjoint with $S$ then $\sqrt[S]{Q}=Q$.
\end{proposition}
\begin{proof}
Let $P$ be a prime ideal disjoint with $S$ containing $I$. If $x\in \sqrt[S]{I}$, there exists $s\in S$ such that $sx\in I$, so $sx\in P$, therefore $x\in P$. It follows that $\sqrt[S]{I}\subseteq\displaystyle\cap_{I\subseteq P, S\cap P= \emptyset} P $. For the converse, suppose to the contrary  that $\sqrt[S]{I}\subsetneq\displaystyle\cap_{I\subseteq P, S\cap P=\emptyset} P $. Let $y\in \cap_{I\subseteq P, S\cap P=\emptyset} P-\sqrt[S]{I}$. Consider the multiplicative subset $S'=\{y^ns\ /\ s\in S, \ n\in \Bbb N\}$. Since $y\not\in \sqrt[S]{I}$ the ideal $I$ is disjoint with $S'$. It follows that there exists a prime ideal $Q$ disjoint with $S'$ containing $I$, in particular $y\not\in Q$, a contradiction. 
\end{proof}
Note that, if $I$ is an ideal of $R$, then $V_S(I)=V_S(\sqrt[S]{I})$. The following result show the influence of the $S$-radical on the $S$-variety and how they are attached.
\begin{proposition}
Let $I$ and $J$ be ideals of $R$. The following statements hold.
\begin{enumerate}
\item $V_S(I)\subseteq V_S(J)$ if and only if  $\sqrt[S]{J}\subseteq \sqrt[S]{I}$.
\item $V_S(I)=V_S(J)$ if and only if $\sqrt[S]{I}=\sqrt[S]{J}$.
\item $V_S(I)=\mathrm{Spec}_SR$ if and only if $\sqrt[S]{I}=\sqrt[S]{0}$.
\item $V_S(I)=\emptyset$ if and only if $ \sqrt[S]{I}=R$ if and only if $I\cap S\ne \emptyset$
\end{enumerate}
\end{proposition}
\begin{proof}
$(1)$ Assume that $V_S(I)\subseteq V_S(J)$. Let $x\in \sqrt[S]{J}$ then $tx^m\in J$ for some $t\in S$ and $m\in \Bbb N$. Suppose to the contrary that $x\not\in \sqrt[S]{I}$, that is, $sx^n\not\in I$ for all $s\in S$ and $n\in \Bbb N$. Let $S'$ be the set of all elements of $R$ of the form $sx^n$ where $s\in S$ and $n\in N$. Then $S'$ is a multiplicative subset of $R$ and $I\cap S'=\emptyset$. So the exists a prime ideal $P$ of $R$ containing $I$ such that $S'\cap P=\emptyset$, in particular $P\in V_S(I)$. But $P\not\in V_S(J)$, in fact, if $P\in V_S(J)$ then $s'J\subseteq P$ for some $s'\in S$, so $s'tx^m\in J$ since $tx^m\in J$, therefor $s'tx^m\in S'\cap P$ which is not possible.  It follows that $x\in \sqrt[S]{I}$. Thus $\sqrt[S]{J}\subseteq \sqrt[S]{I}$. For the converse see \cite{Eda}.
\\
$(2)$ This follows from the previous result.\\
$(3)$ Take $J=(0)$. For $(4)$ take $J=R$. 
\end{proof}
\begin{definition}
Let $R$ be a commutative ring and $S$ be a multiplicative subset of $R$. Let $I$ be an ideal of $R$. We say that $I$ is $S$-radically finite if there exists a finitely generated ideal $J$ such that $\sqrt[S]{I}=\sqrt[S]{J}$.
\end{definition}
\begin{lemma}
An ideal $I$ is $S$-radically finite if and only if there exists a finitely generated sub-ideal $J$ of $I$ such that $\sqrt[S]{I}=\sqrt[S]{J}$.
\end{lemma}
\begin{proof}
If $I$ is $S$-radically finite, then $\sqrt[S]{J}=\sqrt[S]{(a_1,\ldots,a_m)}$ where $a_i\in R$. We have $a_i\in \sqrt[S]{J}$, so $s_ia_i^{n_i}\in J$ for some $s_i\in S $ and $n_i\in \Bbb N$, in particular $\sqrt[S]{(s_1a_1^{n_1},\ldots,s_ma_m^{n_m})}\subseteq \sqrt[S]{J}$. It is easy to see that   $a_i\in \sqrt[S]{(s_1a_1^{n_1},\ldots,s_ma_m^{n_m}) }$. Thus $\sqrt[S]{J}=\sqrt[S]{(s_1a_1^{n_1},\ldots,s_ma_m^{n_m})}$ where $(s_1a_1^{n_1},\ldots,s_ma_m^{n_m})$ is a finitely generated sub-ideal of $J$. The converse is immediate.
\end{proof}
The following theorem is a generalization of \cite[Theorem 6]{Eda}. 
\begin{theorem}
Let $I$ be an ideal of $R$. Then $D_S(I)$ is quasi-compact with respect to the $S$-Zariski topology if and only if $I$ is $S$-radically finite.
\end{theorem}
\begin{proof}
We have $D_S(I)=\cup_{a\in I}D_S(a)$. So if $D_S(I)$ is quasi-compact then $D_S(I)=D_S(a_1)\cup\ldots\cup D_S(a_m)$ for some $a_1,\ldots,a_m\in I$. Thus $D_S(I)=D_S(I_0)$ where $I_0=(a_1,\ldots,a_m)$, it follows from Proposition \ref{Variety} that $\sqrt[S]{I}=\sqrt[S]{I_0}$. Conversely, if $\sqrt[S]{I}=\sqrt[S]{(a_1,\ldots,a_m)}$ for some $a_1,\ldots,a_m\in I$, then $D_S(I)=D_S(a_1,\ldots,a_m)=D_S(a_1)\cup\ldots\cup D_S(a_m)$. It remains to show that each $D_S(a_i)$ is quasi-compact. Assume that $D_S(a_i)=\cup_{b\in C}D_S(b)$ which is an cover of $D_S(a_i)$ by a collection of basic open subsets. Then we have $D_S(a_i)=D_S(I')$ where $I'$ is the ideal generated by all $b\in C$.  This implies that $\sqrt[S]{(a_i)}=\sqrt[S]{I'}$. In particular $a\in \sqrt[S]{I'}$, and this implies that $sa_i=\displaystyle\sum_{k=1}^r\alpha_kb_k$ where $s\in S$, $\alpha_k\in A$ and $b_k\in C$. Thus $ D_S(a_i)=\cup_{k=1}^rD_S(b_k)$. 
\end{proof}
\begin{theorem}
The following statements are equivalents.
\begin{enumerate}
\item $\mathrm{Spec}_SR$ is noetherian with respect to the $S$-Zariski toology..
\item Every ideal of $R$ is $S$-radically finite.
\item For every prime ideal $P$ disjoint with $S$, $P=\sqrt[S]{I_0}$ where $I_0$ is a finitely generated sub-ideal of $P$.
\end{enumerate}
\end{theorem}
\begin{proof}
$(1)\Rightarrow(2)$.  Follows immediately from the previous lemma.\\
$(2)\Rightarrow (3)$. This follows from the fact that, if $P$ is a prime ideal of $R$ disjoint with $S$ then $P=\sqrt[S]{P}$. \\
$(3)\Rightarrow(2)$. Let $\Gamma$ be the set of all ideals $I$ of $R$ disjoint with $S$ such that $\sqrt[S]{I}\ne \sqrt[S]{J}$ for any finitely generated sub-ideal $J$ of $I$. Assume to the contrary that $\Gamma\ne \emptyset$. Let $(J_n)_n$ be an ascending chain of elements of $\Gamma$ and set $J=\cup_nJ_n$. Then $J\in \Gamma$. In fact, if $J\not \in \Gamma$, then $\sqrt[S]{J}=\sqrt[S]{(j_1,\ldots,j_m)}$ for some $j_i\in J$. Since $J=\cup_nJ_n$, $j_1,\ldots,j_m\in J_N$ for some $N\in \Bbb N$, so $\sqrt[S]{J_N}=\sqrt[S]{(j_1,\ldots,j_m)} $ which is not possible. Now, by Zorn's lemma, $\Gamma$ has a maximal element, say $Q$. Now, we will show that $Q$ is a prime ideal of $R$. Assume $ab\in Q$ for some $a,b\in R-Q$. By maximality, $\sqrt[S]{Q+(a)}=\sqrt[S]{(a_1,\ldots,a_m)}$ and $\sqrt[S]{Q+(b)}=\sqrt[S]{(b_1,\ldots,b_n)}$ for some $a_i\in Q+(a)$ and $b_i\in Q+(b)$. It follows that $\sqrt[S]{Q}=\sqrt[S]{(Q+(a))(Q+(b))}=\sqrt[S]{(a_ib_j), i\in \{1,\ldots,m\}, j\in \{1,\ldots,n\}}$  which is not compatible with that fact that $Q\in \Gamma$. Since $Q$ is a prime ideal in particular $S$-prime ideal, a contradiction.\\
$(2)\Rightarrow (1)$ from the previous lemma. 
\end{proof}
\begin{theorem}
$\mathrm{Spec}_SR$ is noetherian with respect to the $S$-Zariski topology if and only if $R$ satisfies the ascending chain condition on the $S$-radical ideals .
\end{theorem}
\begin{proof}
Assume that $\mathrm{Spec}_SR$ is noetherian and let $ (I_n)_n$ be an ascending chain of $S$(radical ideals. Set $I=\cup_nI_n$. Then $I$ is an $S$-radical ideal. Sine $\mathrm{Spec}_S$ is noetherian, $I=\sqrt[S]{(a_1,\ldots,a_n)}$ for some $a_1,\ldots,a_n\in I=\cup_nI_n$. So $a_1,\ldots,a_n\in I_N$ for some $N\in \Bbb N$, it follows that $I_N=I$ that is the chain $I_n$ stabilize. For the converse, assume that $R$ satisfies the ascending chain condition on $S$-radical ideals. Let $I$ be an ideal. The collection of all $S$-radically finite sub-ideal of $\sqrt{I}$ has maximal element, say $J$. The maximality, implies that $\sqrt[S]{I}=J$ and $I$ is $S$-radically finite.   
\end{proof}
\section{\bf The $S$-flat topology}
Let $R$ be a commutative ring. Then there is a unique topology over $\mathrm{Spec}R$ such that the collection of subsets $ V(f)$ with $f\in R$ forms a sub-basis for the opens of this topology. It is called the flat topology. Therefore the collection $V(I)$ where $I$ runs through the set of finitely generated ideals of $R$ forms a basis for the flat opens, for more details see \cite{DoFo},\cite{Tar}. Motivated by this concept, we generalize the flat topology to the $S$-prime spectrum $\mathrm{Spec}_SR$ of a commutative ring $R$. Let $S$ be a multiplicative subset of $R$. Then there is a unique topology over $\mathrm{Spec}_SR$ such that the collection $V_S(f)$ with $f\in R$ forms a sub-basis for the opens of this topology. It is called the $S$-flat topology. Therefore the collection $V_S(I)$ where $I$ runs through the set of finitely generated ideals of $R$ forms a basis for the $S$-flat opens. Note that, if $S\subseteq u(R)$ (consist of units of $R$) then $\mathrm{Spec}_SR=\mathrm{Spec}$ and moreover the $S$-flat topology and the flat topology are the same. 
\begin{remark} If $P$ is an $S$-prime ideal of $R$ then there exists $s_p\in S$ such that $(P:s_p)$ is a prime ideal, moreover, if $(P:s_p')$ is also a prime ideal form some $s_p'\in S$, then $(P:s_p)=(P:s_p')$. In fact, let $a\in (P:s_p)$ then $s_pa\in P\subseteq (P:s_p')$, so $a\in (P:s_p' $ since $a_p\not\in (P:s_p')$, that is $(P:s_p)\subseteq (P:s_p')$. By the same argument, we have $ (P:s_p)=(P:s_p')$.
\end{remark}
\begin{proposition}
The map $\mathrm{Spec}_SR\to \mathrm{Spec}R$, $P\mapsto (P:s_p)$ is continuous with respect to the $S$-flat topology and flat topology.
\end{proposition}
\begin{proof}
Let $a\in R$. We have $\{P\in \mathrm{Spec}_SR\ / (P:s_p)\in V(a)\}=V_S(a)$. In fact, if $(P:s_p)\in V(a)$ then $s_pa\in P$, so $P\in V_S(a)$. Conversely, if $P\in V_S(a)$ then $s'a\in P$ for some $s'\in S$, so $s'a\in (P:s_p)$, hence $a\in (P:s_p)$. Thus $(P:s_p)\in V(a)$.
\end{proof}
\begin{proposition}
Let $\varphi:R\to R'$ be a morphism of rings such that $0\not\in \varphi(S)$. Then the induced morphism $\varphi_S:\mathrm{Spec}_SR'\to \mathrm{Spec}_SR$ is continuous with respect the $S$-flat and $\varphi(S)$-flat topologies.
\end{proposition}
\begin{proof}
If $I$ is a finitely generated ideal of $R$ then by \cite[Theorem 3]{Eda} $\varphi_S^{-1}(V_S(I))=V_{\varphi(S)}(\varphi(I)R')$. Since $\varphi(I)R'$ is a finitely generated ideal of $R'$, $V_{\varphi(S)}(\varphi(I)R') $ is open with respect to the $S$-flat topology. Therefore, $\varphi_S$ is continuous with respect to the $S$-flat and $\varphi(S)$-flat topologies.  
\end{proof}
\section{\bf Compactness and irreducibility }
We start this section with the following useful lemma.
\begin{lemma}
Let $R$ be a commutative ring and $S$ be a multiplicatively closed subset of $R$. If $I,J$ are ideals of $R$ then $\sqrt[S]{\sqrt[S]{I}\sqrt[S]{J}}=\sqrt[S]{IJ}$. 
\end{lemma} 
\begin{proof}
Since $IJ\subseteq \sqrt[S]{I}\sqrt[S]{J}$, we get $\sqrt[S]{IJ}\subseteq \sqrt[S]{\sqrt[S]{I}\sqrt[S]{J}}$. Let $x\in \sqrt[S]{I}\sqrt[S]{J}$, then $x=\displaystyle\sum_{k=1}^ra_kb_k $ where $a_k\in \sqrt[S]{I}$ and $b_k\in\sqrt[S]{J}$. For each $k$, $s_ka_k^{n_k}\in I$ and $t_kb_k^{m_k}\in J$ for some $s_k,t_k\in S$ and $n_k,m_k\in \Bbb N$. We have $t_ks_k(a_kb_k)^{n_k+m_k}\in IJ$, so $a_kb_k\in \sqrt[S]{IJ}$. It follows that $x\in \sqrt[S]{IJ}$. Consequently, $\sqrt[S]{I}\sqrt[S]{J}\subseteq  \sqrt[S]{IJ}$. Thus $\sqrt[S]{\sqrt[S]{I}\sqrt[S]{J}}=\subseteq \sqrt[S]{IJ}$.
\end{proof}
\begin{theorem}\label{Compact}
Let $I$ be an ideal of $R$. Then $V_S(I)$ is quasi-compact  subset of $\mathrm{Spec}_SR$ with respect to the $S$-flat topology, that is, every closed subset with respect to the $S$-Zariski topology is a quasi-compact subset with respect to the $S$-flat topology. 
\end{theorem}
\begin{proof}
If $I\cap S\ne \emptyset$ then $V_S(I)=\emptyset$ is quasi-compact.  For $I\cap S=\emptyset$ assume that  $V_S(I)\subseteq \cup_{J\in K}V_S(J)$ which  is an $S$-flat open cover of $V_S(I)$ where each $J\in K$ is a finitely generated ideal of $R$. Let $K'$ be the set of all ideals of the form $\prod_{J\in K_0}J$ where $K_0$ is a finite subset of $K$. The  product of two elements of $K'$ is an element of $K'$ and every element of $K'$ is finitely generated, moreover $\cup_{J\in K}V_S(J)=\cup_{J\in K'}V_S(J)$ in fact since $K\subseteq K'$, we have $\cup_{J\in K}V_S(J)\subseteq \cup_{J\in K'}V_S(J)$. For $J'\in K'$ we have $J'=J_1\ldots J_r$ where $J_i\in K$, so $V_S(J')=V_S(J_1)\cup \ldots\cup V_S(J_r)\subseteq \cup_{J\in K}V_S(J)$. Now, we show that $V_S(I)\subseteq V_S(J)$ for some $J\in K'$. Suppose to the contrary that $V_S(I)\not\subseteq V_S(J)$ for all $J\in K'$, in term of $S$-radical, $\sqrt[S]{J}\not\subseteq \sqrt[S]{I}$ for all $J\in K'$. Consider
$$\Gamma=\{ L\ / L \text{ ideal of } R, L\cap S=\emptyset, \text{ such that for all } J\in K', \sqrt[S]{J}\not\subseteq \sqrt[S]{L}\}$$ 
Note that $\Gamma$ is a nonempty set since $I\in \Gamma$. Let $(L_n)_n$ be a ascending chain of elements of $\Gamma$ and set $L=\cup_nL_n$. Clearly $L\cap S=\cup_n(L_n\cap S)=\emptyset$. If $L\not\in \Gamma$  then $\sqrt[S]{J}\subseteq  \sqrt[S]{L}$. Since $J$ is finitely generated, $J=(j_1,\ldots,j_m)$ for some $j_1,\ldots,j_m\in J$. For each $1\le i\le m$, $s_ij_i^{n_i}\in L$ for some $s_i\in S$ and $n_i\in \Bbb N$. Since $(L_n)_n$ is an ascending chain, there exists an $L_{N}$ containing all $s_ij_i^{n_i}$, where $1\le i\le m$. So each $j_i$ belong to $\sqrt[S]{L_{N}}$, hence $\sqrt[S]{J}\subseteq \sqrt[S]{L_{N}}$ with is not compatible with the fact that $L_{N}\in \Gamma$. It follows that $L\in \Gamma$. By Zorn's lemma, $\Gamma$ has a maximal element, say $P$. Now, we show that $P$ is a prime ideal disjoint with $S$ and in particular $S$-prime ideal. Assume that $ab\in P$ with $a,b\not\in P$. Consider the two ideals $P_a=P+(a)$ and $P_b=P+(b)$. Since $P\subsetneq P_a$ and $P\subsetneq P_b$, we get  $P_a,P_b\not\in \Gamma$. So $\sqrt[S]{J_1}\subseteq \sqrt[S]{P_a}$ and $\sqrt[S]{J_2}\subseteq \sqrt[S]{P_b}$ for some $J_1,J_2\in K'$, hence $\sqrt[S]{J_1}\sqrt[S]{J_2}\subseteq \sqrt[S]{P_a}\sqrt[S]{P_b}$, now $\sqrt[S]{\sqrt[S]{J_1}\sqrt[S]{J_2}}\subseteq \sqrt[S]{\sqrt[S]{P_a}\sqrt[S]{P_b}} $, from the previous lemma,  $\sqrt[S]{J_1J_2}\subseteq \sqrt[S]{P_aP_b}$. It is easy to see that $J_3=J_1J_2\in K'$ and $\sqrt[S]{P_aP_b}\subseteq \sqrt[S]{P^2 }=\sqrt[S]{P}$, that is $\sqrt[S]{J_3}\subseteq \sqrt[S]{P}$ where $J_3\in K'$ which is not compatible with $P\in \Gamma$. As a consequence, $P$ is a prime ideal disjoint with $S$, so an $S$-prime ideal of $R$. Now, we see that $P\in V_S(I)$ but $P\not\in \cup_{J\in K'}V_S(J)$ a contradiction.  There exists $J\in K'$ such that $V_S(I)\subseteq V_S(J)$. Since $J\in K'$, $J=J_1\ldots J_r$ for some $J_1,\ldots,J_r\in K$, so $V_S(I)\subseteq V_S(J_1)\cup \ldots\cup V_S(J_r)$ as desired. 
\end{proof}
\begin{corollary}\label{qcomact}
\begin{enumerate} 
\item The quasi-compact open subsets of $\mathrm{Spec}_SR$ with respect to the $S$-flat topology are $V_S(I)$ where $I$ is a finitely generated ideal of $R$. 
\item $\mathrm{Spec}_SR$ is quasi-compact with respect to the $S$-flat topology.
\item Let $I$ be a ideal of $R$. Then $V_S(I)$ is an open with respect to the $S$-flat topology if and only if $\sqrt[S]{I}=\sqrt[S]{J}$ for some finitely generated ideal $J$.
\end{enumerate}
\end{corollary}
\begin{proof}
\begin{enumerate}
\item If $I$ is a finitely generated ideal of $R$ then $V_S(I)$ is quasi-compact and open  with respect to the $S$-flat topology. Conversely, let $U$ be a quasi-compact open with respect to the $S$-flat topology. Since $U=\cup_{t}V_S(I_t)$ where $I_t$ are finitely generated ideals and $U$ quasi-compact, it follows that $U=\cup_{1\le k\le r}V_S(I_{t_i})=V_S(\prod_{1\le k\le r}I_{t_i})$. The result follows from the fact that $\prod_{1\le k\le r}I_{t_i} $ is finitely generated.
\item No comment.
\item If $ \sqrt[S]{I}=\sqrt[S]{J}$ for some finitely generated ideal $J$ then $$V_S(I)=V_S(J)$$ is  a quasi-compact open with respect to the $S$-flat topology. If $V_S(I)$ is an open with respect to the $S$-flat topology, since it is a closed with respect to the $S$-Zariski topology, there exists a finitely generated ideal $J$ such that $V_S(I)=V_S(J)$. Thus $\sqrt[S]{I}=\sqrt[S]{J}$.  
\end{enumerate}
\end{proof}
For $P\in \mathrm{Spec}_SR$, denote $\Lambda(P)$ the closer of the point $P$ with respect to the $S$-flat topology.
\begin{proposition}
For $P\in \mathrm{Spec}_SR$ we have $\Lambda(P)=\{Q\in \mathrm{Spec}_SR\ / sQ\subseteq P \text{ for some } s\in S\}$
\end{proposition}
\begin{proof}
Let $Q\in \Lambda(P)$. Since $V_S(Q)$ is a neighborhood of $Q$, we have $P\in V_S(Q)$, so $sQ\subseteq P$ for some $s\in S$. Conversely, let $Q\in \mathrm{Spec}_SR$ with $sQ\subseteq P$ for some $s\in S$. Now, take a basis open $V_S(I)$ of $\mathrm{Spec}_SR$ containing $Q$, then $tI\subseteq Q$ for some $t\in S$, so $stI\subseteq sQ\subseteq P $, that is, $P\in V_S(I)$. It follows that every open neighborhood of $Q$ contain $P$. Thus $Q\in \Lambda(P)$.  
\end{proof}

\begin{theorem}\label{Irreducible}
Avery irreducible closed subset with respect to the $S$-flat topology has a generic point. 
\end{theorem}
\begin{proof}
Let $F$ be an irreducible closed subset of $\mathrm{Spec}_SR$ with respect to the $S$-flat topology. Let $Y$ be the set of all elements $f\in R$ such that $F\cap V_S(f)\ne \emptyset$. Then $\cap_{f\in Y}V_S(f)\ne \emptyset$.  In fact if $\cap_{f\in Y}V_S(f)=\emptyset$, then $\cup_{f\in Y}D_S(f)=\mathrm{Spec}_SR $, so $D_S(I)=\mathrm{Spec}R$ where $I=(Y)$ the ideal generated by $Y$. It follows that $s=a_1f_1+\ldots+a_mf_m$ for some $s\in S$, $a_i\in R$ and $f_i\in Y$, that is, $D_S(f_1)\cup\ldots\cup D_S(f_m)=\mathrm{Spec}_SR$, hence $\cap_{i=1}^mV_S(f_i)=\emptyset $, in particular $\cap_{i=1}^m(F\cap V(f_i))=\emptyset$ a contradiction since it is a finite intersection of a nonempty open subsets of $F$ which is irreducible. Now, we show that $F\cap\left( \cap_{f\in Y}V_S(f)\right)\ne\emptyset$. Suppose to the contrary that $ F\cap\left( \cap_{f\in Y}V_S(f)\right)=\emptyset$, that is, $V_S(I)\subseteq \mathrm{Spec}_SR-F$ where $I=(Y)$ the ideal generated by $Y$. Since $F$ is $S$-flat closed, $\mathrm{Spec}_SR-F=\cup_{J\in K}V_S(J)$  where $K$ is a subset of finitely generated ideals of $R$. We have $V_S(I)\subseteq \cup_{J\in K}V_S(J)$ and by Theorem \ref{Compact}   $V_S(I)$ is a quasi-compact subset with respect to th flat topology , so $V_S(I)\subseteq \cup_{J\in K_0}V_S(J)$ where $K_0\subseteq K$ is a finite subset of finitely generated ideals. Now, take $J_1=\prod_{J\in K_0}J$, then $J_1$ is a finitely generated ideal and $V_S(I)\subseteq V_S(J_1)$, that is , $D_S(J_1)\subseteq D_S(I)=\cup_{f\in Y}D_S(f)$. Since $J_1$ is finitely generated, $D_S(J_1)$ is quasi-compact with respect to the $S$-Zariski topology, so $D_S(J_1)\subseteq D_S(f_1)\cup\ldots\cup D_S(f_r)$ where $f_1,\ldots,f_r\in Y$. It follows that $F\subseteq D_S(J_1)\subseteq D_S(f_1)\cup\ldots\cup D_S(f_r)$, that is, $V_S(f_1)\cap\ldots\cap V_S(f_r)\cap F=\emptyset$ a contradiction since $F$ is irreducible. Thus $F\cap \left( \cap_{f\in Y}V_S(f)\right)\ne \emptyset $. Pick $P\in F\cap\left( \cap_{f\in Y}V_S(f)\right)$, then $\Lambda(P)=F$. 
\end{proof}
\begin{remark} If $P\in \mathrm{Spec}_SR$ and $s\in S$ then $sP\in \mathrm{Spec}_SR$ and it is easy to see that $\Lambda(sP)=\Lambda(P)$. It follows that generic point is not unique. But if we restrict to prime ideal we have the uniqueness, as in the following result.  
\end{remark}
\begin{lemma}\label{Unique}
If $p,q\in \mathrm{Spec}_SR$ are prime ideals. Then $\Lambda(p)=\Lambda(q)$ if and only if $p=q$. 
\end{lemma}
\begin{proof}
If $\Lambda(p)=\Lambda(q)$ then $sp\subseteq q$ and $s'q\subseteq p$ for some $s,s'\in S$, so $p\subseteq q$ and $q\subseteq q$, thus $p=q$.
\end{proof}
\begin{proposition}
The map $P\mapsto \Lambda(P)$ is bijection between $\mathrm{Spec}S^{-1}R$ and the irreducible closed subsets of $\mathrm{Spec}_SR$ with respect to the $S$-flat topology Moreover, under this correspondence, the set of maximal ideals of $S^{-1}R$ is in bijection with the set of $S$-flat irreducible components of $\mathrm{Spec}_SR$.   
\end{proposition}
\begin{proof}
Let $F$ be an irreducible closed subset with respect to the $S$-flat topology, from Theorem \ref{Irreducible}, $F=\Lambda(P)$ for some $S$-prime ideal $P$. The result follows from the fact that $(P:s_p)$ is a prime ideal disjoint with $S$  for some $s_p\in S$ and $\Lambda(P)=\Lambda(P:s_p)$. The uniqueness flows from the Lemma \ref{Unique}. Moreover, if $P$ is a maximal ideal of $R$ with respect to disjoint with $S$ (that is maximal ideal of $S^{-1}R$) and $\Lambda(P)\subseteq \Lambda(Q)$ for some prime ideal $Q\in \mathrm{Spec}_SR$. Then $sP\subseteq Q $, so $P\subseteq Q$, thus $P=Q$, hence $\Lambda(P)=\Lambda(Q)$. Now, assume that $\Lambda(P)$ is an irreducible component where $P\in \mathrm{Spec}_SR$ is a prime ideal. Let $M$ be a prime ideal containing $P$ and disjoint with $S$. We have $\Lambda( P)\subseteq \Lambda(M)$. Since $\Lambda(P)$ is an irreducible component, we get $ \Lambda( P)= \Lambda(M)$, it follows from the Lemma \ref{Unique}, that $P=M$.  
\end{proof}
\begin{corollary}
The following statements are equivalent.
\begin{enumerate}
\item $\mathrm{Spec}_SR$ with respect to the $S$-flat topology is $T_0$.
\item  Every $S$-prime ideal is a prime ideal.
\end{enumerate}
\end{corollary}
\begin{proof}
$(1)\Rightarrow(2)$ If $P$ is an $S$-prime ideal then $\Lambda(P)=\Lambda(P:s_p)$ where $s_p\in S$ such that $(P:s_p)$ is a prime ideal. Then $P=(P:s_p)$ is a prime ideal.\\
$(2)\Rightarrow(1)$ If $\Lambda(P)=\lambda(Q)$ with $P,Q\in \mathrm{Spect}_SR$ which are prime ideals, then by Lemma \ref{Unique}, $P=Q$, so $\mathrm{Spec}_SR$ is $T_0$.
\end{proof}
 \section{\bf Connectivity}
 \begin{theorem}
 Assume that $\mathrm{Spec}_SR=C_1\cup C_2$ with $C_1\cap C_2=\emptyset$.The following statements are equivalent.
 \begin{enumerate}
 \item $C_1$, $C_2$ are closed with respect to the $S$-Zariski topology.
 \item $C_i=V_S(f_i)$ for $i=1,2$, such that $f_1+f_2\in S$ and $ f_1f_2\in \sqrt[S]{0}$.
 \item $C_1$, $C_2$ are closed with respect to the $S$-flat topology.
 \end{enumerate}
 \end{theorem}
 \begin{proof}
 $(1)\Rightarrow(2)$ Let $I,J$ be ideals of $R$ with $C_1=V_S(I)$ and $C_2=V_S(J)$.  Since $C_1\cup C_2=V(IJ)=\mathrm{Spec}_SR$ and $C_1\cap C_2=V_S(I+J)=\emptyset$, we have  $\sqrt[S]{IJ}=\sqrt[S]{(0)}$ and $(I+ J)\cap S\neq\emptyset $, so $s=f_1+f_2\in S$ for some $f_1\in I$ and $f_2\in J$. Since $f_1f_2\in IJ$  and $\sqrt[S]{IJ}=\sqrt[S]{(0)}$, it follows that $f_1f_2\in \sqrt[S]{0}$. Clearly, $V_S(f_1)\subseteq V_S(I)$, $V_S(f_2)\subseteq V_S(J)$, $V_S(f_1)\cup V_S(f_2)=\mathrm{Spec}_SR$ and $V_S(f_1)\cap V_S(f_2)=\emptyset$, then $V_S(f_1)=V_S(I)=C_1$ and $V_S(f_2)=V_S(J)=C_2$.\\
 $(2)\Rightarrow(3)$ In this cases, $V_S(f_1)=\mathrm{Spec}_SR-V_S(f_2)$ and $V_S(f_2)=\mathrm{Spec}_SR-V_S(f_1)$, so $V_S(f_1)$ and $V_S(f_2)$ are closed with respect to the $S$-flat topology.\\ 
 $(3)\Rightarrow(1)$ If $C_1,C_2$ are closed with respect to the $S$-flat topology, then they are opens with respect to the $S$-flat topology. So $C_1=\cup_{I\in K_1}V_S(I)$ and $C_2=\cup_{I\in K_2}V_S(I)$ where $K_1,K_2$ are subsets of finitely generated ideals. Since $\mathrm{Spec}_SR=\left(\cup_{I\in K_1}V_S(I)\right)\cup \left(\cup_{I\in K_2}V_S(I)\right)$ and $\mathrm{Spec}_SR$ is quasi-compact with respect to the $S$-flat topology,  we get $\mathrm{Spec}_SR= \left(\cup_{I\in K_1'}V_S(I)\right)\cup\left(\cup_{I\in K_2'}V_S(I)\right)$ where $K_1'\subseteq K_1$ and $K_2'\subseteq K_2$ are finite subsets. Set $I_1=\prod_{I\in K_1'}I$ and $I_2=\prod_{I\in K_2'}$, then $ \mathrm{Spec}_SR=V_S(I_1)\cup V_S(I_2)$ and $V_S(I_1)\cap V_S(I_1)=\emptyset $. Thus $C_1=V_1(I_1)$ and $C_2=V_S(I_2)$ and $C_1,C_2$ are closed with respect to the $S$-Zariski topology.
 \end{proof}
 \begin{corollary}
 $\mathrm{Spec}_SR$ is connected with respect to the $S$-flat topology if and only if it is connected with respect to the $S$-zariski topology.
 \end{corollary}
 \begin{proof}
 The proof is immediate from the previous Theorem. 
 \end{proof}
 \section{\bf Noetherianess }
 \begin{lemma}\label{Irre}
 Let $I$ be an ideal of $R$ and $P\in D_S(I)$. Then $\Lambda(P)\subseteq D_S(I)$.
 \end{lemma}
 \begin{proof}
 Let $Q\in \Lambda(P)$, then $sQ\subseteq p$ for some $s\in S$. If $Q\in V_S(I)$, then $ tI\subseteq Q$ for some $t\in S$, so $tsI\subseteq sQ\subseteq P$ a contradiction. Thus $Q\in D_S(I)$. 
 \end{proof}
 \begin{theorem}
 The following statements are equivalent.
 \begin{enumerate}
 \item $\mathrm{Spec}_SR$ is a noetherian space with respect to the $S$-flat topology.
 \item The opens of $\mathrm{Spec}_SR$ with respect to the $S$-flat topology are $V(I)$ where $I$ runs through the set of  finitely generated ideals of $R$.
 \item For every prime ideal $P$ disjoint with $S$, $\Lambda(P)=D_S(f)$ for some $f\in R$. 
 \item Any arbitrary intersection of $S$-Zariski open subsets is an $S$-Zariski open.
 \end{enumerate}
 \end{theorem}
 \begin{proof}
 $(1)\Rightarrow (2)$. Is immediate from the Corollary \ref{qcomact}.\\
 $(2)\Rightarrow(3)$. Since $\Lambda(P)$ is closed with respect to the $S$-flat topology, $\Lambda(P)=D_S(I)=D_S(f_1)\cup\ldots\cup D_S(f_r)$ where $I=(f_1,\ldots,f_r)$ is a finitely generated ideal of $R$. But $\Lambda(P)$ is an irreducible closed subset with respect to the $S$-flat topology and each $D_S(f_)$ is closed with respect to the $S$-flat topology, so $\Lambda(P)=D_S(f_i)$ for some $f_i$.\\
$(3)\Rightarrow(4)$. Let $(D_S(I_k))_k$ be a collection of $S$-Zariski open subsets and $D=\cap_kD_S(I_k)$. If $D$ is empty then it is an $S$-Zariski open. If $D$ is a nonempty set and  $Q\in D$, then $Q\in D_S(I_k)$ for all $k$, by Lemma \ref{Irre},  $\Lambda(P)\subseteq D_S(I_k)$ for all $k$. Thus $\Lambda(P)\subseteq D$. It follows that $D=\cup_{P\in D}\Lambda(P)$ is an $S$-Zariski open since by hypothesis each $\Lambda(P)$ is an $S$-Zariski open.\\
$(4)\Rightarrow(1)$. Let $U=\cup_{I\in K}V(I)$ be an $S$-flat open where each $I\in K$ is a finitely generated ideal of $R$. Since $U$ is $S$-Zariski closed as union of Zariski closed, $U=V(J)$ for some ideal $J$. By the Corollary \ref{qcomact}, $U$ is quasi-compact.    
 \end{proof}
 \section{\bf A question about flat morphism }
 It is well known that if $R_1\to R_2$ is a flat morphism then the going-down property holds. To realize the $S$-flat topology from flat morphisms, this depend on the following question.\\  
 \textbf{Question :}  Let $f:R_1\to R_2$ be a morphism of rings and $S$ be multiplicatively closed subset of $R_1$ such that $0\not\in f(S)$. Is the going-down property holds for $S$-prime ideals?

\end{document}